\documentclass[11pt]{amsart}
\usepackage{amsmath,amssymb,amscd,mathrsfs}
\usepackage{url}
\usepackage{eepic,epic}
\usepackage{longtable}
\newtheorem{theorem}{Theorem}[section]

\newtheorem{lemma}[theorem]{Lemma}
\newtheorem{corollary}[theorem]{Corollary}
\theoremstyle{definition}
 
\newtheorem{remark}[theorem]{Remark}

\newtheorem{example}[theorem]{Example}

\newcommand{\tdfn}{f_N}
\newcommand{\tdfm}{f_M}
\newcommand{\pphi}{\varphi}
\newcommand{\hhh}{p}

\def\ord{\mathop{\mathrm{ord}}\nolimits}

\def\C{\mathbb C} \def\Q{\mathbb Q} \def\R{\mathbb R} \def\Z{\mathbb Z}
 \def\F{\mathbb F}

\DeclareMathOperator{\rank}{rank}
\DeclareMathOperator{\lcm}{lcm}
\newcommand{\divides}{\mid}

\newenvironment{(enumerate)}{
  \begin{enumerate}
  
  }{\end{enumerate}}

\begin{document} 
\title[Automorphisms of Enriques surfaces and their entropy]{On automorphisms of 
Enriques surfaces and their entropy}
\author[Y. Matsumoto]{Yuya Matsumoto}
\address{Graduate School of Mathematics, Nagoya University, 
Furocho, Chikusaku, Nagoya, 464-8602, Japan}
\email{matsumoto.yuya@math.nagoya-u.ac.jp}
\author[H. Ohashi]{Hisanori Ohashi}
\address{Department of Mathematics, 
Faculty of Science and Technology, 
Tokyo University of Science, 
2641 Yamazaki, Noda, 
Chiba 278-8510, Japan}
\email{ohashi\_hisanori@ma.noda.tus.ac.jp, ohashi.hisanori@gmail.com}
\author[S. Rams]{S{\L}awomir Rams}
\address{Institute of Mathematics, Jagiellonian University, 
ul. {\L}ojasiewicza 6,  30-348 Krak\'ow, Poland}
\email{slawomir.rams@uj.edu.pl}
\thanks{Y.~M. is supported by JSPS KAKENHI Grant Numbers 15H05738 and 16K17560.
H.~O. is supported by JSPS KAKENHI 15K17521.
S.~R. is partially supported  by National Science Centre, Poland, grant 2014/15/B/ST1/02197.
}

\begin{abstract} 
Consider an arbitrary automorphism of an Enriques surface 
with its lift to the covering $K3$ surface.
We prove a bound of the order of the lift acting
on the anti-invariant cohomology sublattice of the Enriques involution.
We use it to obtain some mod 2 constraint on the original automorphism.
As an application, we give a necessary condition for Salem numbers to be dynamical degrees on Enriques surfaces and obtain a new lower bound on the
minimal value.
In the Appendix, we give a complete list of Salem numbers 
that potentially may be the minimal dynamical degree on Enriques surfaces and 
for which the existence of geometric automorphisms is unknown.
\end{abstract}
\date{\today} 
\maketitle

\section{introduction}

It is known that the only compact K\"{a}hler surfaces that admit automorphisms 
of positive topological entropy are rational, Enriques, $K3$ surfaces and
 complex tori (see e.g. \cite[$\S$~2.5]{cantat}). 
Salem numbers that can be realized  as 
the dynamical degrees of
automorphisms of 2-dimensional tori are fully characterized in \cite[Thm~1.1]{reschke12}
in terms of values of the minimal polynomials. The same question is solved for 
rational surfaces in terms of Weyl groups in \cite{Uehara}.
These are exactly the description of the dynamical spectrum 
\[\Lambda(\mathcal{C})=\{\lambda(f)\in\mathbb{C}\mid \lambda(f)\text{is the dynamical degree of $f\in \mathrm{Aut}(S)$ for some $S\in \mathcal{C}$}\}\]
where the class of surfaces $\mathcal{C}$ is taken to be $2$-tori or rational surfaces.
For $K3$ surfaces, the recent preprint \cite{BF-T} describes the case of degree 22 Salem numbers. Other degrees on $K3$ surfaces and also on Enriques surfaces the description 
of $\Lambda(\mathcal{C})$ remains open.

The purpose of this note is to give a new property which is satisfied by all
automorphisms of Enriques surfaces. As a consequence, we obtain a new constraint on 
the Salem numbers that appear as 
the dynamical degrees of automorphisms of Enriques surfaces, namely a property of $\Lambda
(\text{Enriques})$. It should be noted 
that despite its ergodic
interpretation \cite[$\S$.2.2.2]{cantat}, the problem we consider is purely algebraic, 
in the sense that the dynamical degree of an automorphism of an Enriques surface $S$ can be detected as the
spectral radius of the map it induces on $\mathrm{Num}(S)$ (see $\S$.\ref{sect-dynamical-degrees}).

To state the theorem, 
let $S$ be an Enriques surface and let  $\tilde{S}$ be its $K3$-cover. 
We denote by
$\varepsilon$ the covering involution of the double \'etale cover $\pi: \tilde{S} \rightarrow S$ and put $N$ to denote 
the  orthogonal complement of the $\varepsilon$-invariant sublattice $H^2(\tilde{S},\Z)^{\varepsilon}$ in the 
lattice $H^2(\tilde{S},\Z)$:
\begin{equation} \label{eq-latN}
N = (H^2(\tilde{S},\Z)^{\varepsilon})^{\perp} \, .  
\end{equation}
Recall that  for an arbitrary automorphism 
 $f\in \mathrm{Aut}(S)$, there exists  a lift $\tilde{f} \in \mathrm{Aut}(\tilde{S})$. Obviously the lift in question is not unique 
(given $\tilde{f}$, the automorphism 
$\tilde{f} \circ \varepsilon$ is also a lift of $f$), but the constraints we prove are valid for any choice of $\tilde{f}$.
As is well-known (see e.g.  {\bf \cite{namikawa}}), the lattice $N$  is stable under the cohomological action $\tilde{f}^*$, hence the restriction 
$$
\tdfn := \tilde f^* \rvert_N
$$
is an automorphism (isometry) of $N$.  It is easy to see that  the order $\ord(\tdfn)$ is finite, 
Lemma \ref{Nfin}.  Here we show a more precise constraint on 
the order of the map $\tdfn$:
\begin{theorem} \label{th1}
Let $S$ be an Enriques surface and let $f\in \mathrm{Aut}(S)$. 
Then, the order of $\tdfn$ is an integer which divides at least one of the integers
\[120, 90, 84, 72, 56, 48.\]
Equivalently and explicitly, these are one of the 31 integers
\begin{equation} \label{eq-T1}
120,90,84,72,60,56,
 48,45,42,40,36,30,28, 24,21,20,18, 16, 15,14,12,10,\dots,1.
\end{equation}
\end{theorem}

We use the above theorem to derive the following mod 2 constraint for a Salem number 
to be the dynamical degree of automorphisms of Enriques surfaces  (which refines \cite[Lemma~4.1]{third}):
\begin{theorem} \label{mod2}
Let $f$ be an automorphism of an Enriques surface $S$ and let $s_{\lambda}$ be the minimal polynomial of its dynamical degree $\lambda(f)$. 
Then the modulo $2$ reduction of 
$s_{\lambda}$ is a product of (some of) the following polynomials
\begin{align*}
F_{ 1}(x) &=                                                          x + 1, \quad 
F_{ 3}(x)  =                                                    x^2 + x + 1, \quad 
F_{ 5}(x)  =                                        x^4 + x^3 + x^2 + x + 1, \\
F_{ 7}(x) &=                            x^6 + x^5 + x^4 + x^3 + x^2 + x + 1, \\
F_{ 9}(x) &=                            x^6 +             x^3           + 1, \\
F_{15}(x) &=                x^8 + x^7 +       x^5 + x^4 + x^3 +       x + 1, \\
\end{align*}
\end{theorem}


\noindent
Here each $F_m(x) \in \F_2[x]$ is the modulo $2$ reduction of the $m$-th cyclotomic polynomial $\Phi_m(x)\in \mathbb{Z}[x]$.
Among these 
six polynomials, 
$F_7(x)$ and $F_{15}(x)$ are products of two distinct irreducible factors, each of which is  not self-reciprocal,
whereas the other four are irreducible.

We do not know whether all the  six         
factors above do appear among factorizations
of the minimal polynomial of the map induced by an automorphism of Enriques surfaces, but we are able to give examples where
$F_1(x), F_3(x), F_5(x)$ 
do come up (see Example~\ref{example-existence-3factors}.a). 
On the other hand, we can check that they all appear from some {\em{lattice isometry}}
of $U\oplus E_8$ by using lattice theory and ATLAS table, for example.

A closely related problem to the description of the dynamical spectrum is to find the minimal nontrivial dynamical degree $1\neq \lambda\in \Lambda(\mathcal{C})$.

This question was answered for complex tori (see \cite[Thm~1.3]{mcmullen11}), rational and   $K3$ surfaces
by McMullen (see  
\cite{mcmullen07,mcmullen11,mcmullen16}), whereas
the smallest dynamical degree attained by automorphisms of Enriques surfaces is yet to be found (\cite[Question~4.5.(3)]{third}).
By  \cite[Remark~4.4]{third},  none of the smallest five Salem numbers (including the ones
of degree $>10$) can be realized on Enriques surfaces. 
Although explicit 
descriptions of automorphism groups of several special families (see \cite{BP,MO15}) of Enriques surfaces are known,
our present knowledge seems not to be enough to 
determine the minimal dynamical degree of automorphisms of surfaces in this class. Presently the 
smallest known dynamical degree of an automorphism of an Enriques surface is the one 
constructed by Dolgachev  \cite[Table~2]{dolgachev16}, who found an 
automorphism of an Enriques surface  (of Hesse type) of dynamical degree $\lambda_D=2.08101\ldots$
(see Example~\ref{Dolgachevexample} for more details).  
As to this respect, in Example \ref{MO15family}, we give an additional study on the 
family of \cite{MO15} to show that all nontrivial dynamical degrees of automorphisms of Enriques surfaces in \cite{MO15} are at least $\lambda_D$. The result thus fails to give 
a new lower bound, but gives a good account for what is going on.

To fill the gap, the constraint given by Thm~\ref{mod2} works 
to give the following slightly better theoretical lower bound:
\begin{corollary} \label{th2}
The dynamical degree of an automorphism of an Enriques surface is greater than or equal to
the Salem number $\lambda=1.35098\cdots$ given by  the polynomial
\[x^{10}-x^9-x^6+x^5-x^4-x+1.\]
\end{corollary}
We use Theorem~\ref{mod2} together with \cite{Gross--McMullen}, combined with Dolgachev's example \cite[Table~2]{dolgachev16}
to show that the smallest (non-trivial) dynamical degree
of an automorphism of an Enriques surface must be one of the 39 Salem numbers 
which we list in $\S.$\ref{sect-appendix}~Appendix.

Our approach is inspired by Oguiso's proof of \cite[Thm~1.2]{third}. 
We consider mod 2 reduction of the cohomological action, and apply results from \cite{Gross--McMullen} to obtain a detailed picture.
It should be noted that our approach does rule out numerous Salem numbers (see Example~\ref{rem-possibilities}), but 
Thm~\ref{mod2} cannot  lead to a necessary and sufficient condition 
for a Salem number to be the  dynamical degree of an automorphism  of an Enriques surface.
One possible way of determining the exact minimal value  of non-trivial dynamical degrees
of automorphisms of Enriques surfaces will be that the 
remaining $39$ cases can be treated efficiently by some refinement of McMullen's method \cite{mcmullen16},
but this task exceeds the scope of this paper.

\noindent
{\bf Convention:} 
In this note
we work over the field of complex numbers $\C$. 

\section{Proof of Theorem \ref{th1}} \label{sect-th1}

We maintain the notation of the introduction: $S$ is assumed to be  an Enriques surface and $f\in \mathrm{Aut}(S)$ is an automorphism. 
As is well-known, the canonical cover 
$\tilde{S}=\mathrm{Spec}(\mathcal{O}\oplus \mathcal{O}(K_S))$ is a $K3$ surface
and the morphism $\pi: \tilde{S} \rightarrow S$  is a double \'etale cover. 
We denote 
by $\varepsilon$ the covering involution of $\pi$. 
Since the automorphism $f$ preserves the canonical class $K_S\in \mathrm{Pic}(S)$, $f$ lifts 
to an automorphism $\tilde{f}$ of $\tilde{S}$. 

Recall that for an Enriques surface the lattice $\mathrm{Num}(S)$ is  the free 
part of the cohomology group  $H^2(S,\Z)$. Let $$M:=H^2(\tilde{S},\Z)^{\varepsilon}$$ be the $\varepsilon$-invariant sublattice of the cohomology lattice $H^2(\tilde{S},\Z)$
and let $N:= M^\perp$
be its orthogonal complement. The direct orthogonal sum 
$M \oplus N$ is a finite index sublattice of the  lattice $H^2(\tilde{S},\Z)$.
Moreover,  we know by 
\cite[Proposition (2.3)]{namikawa} that 
$M$ coincides with the pullback of $H^2(S,\Z)$ by $\pi$, 
hence we have the isomorphisms
\begin{equation} \label{eq-enriques-isomorphisms}
M \simeq \mathrm{Num}(S)(2) \simeq U(2)\oplus E_8(2) \text{ and } N\simeq U\oplus U(2)\oplus E_8(2),
\end{equation}
where $U$ denotes the unimodular hyperbolic plane and $E_8$ is the unique even unimodular
negative-definite lattice of rank 8. Moreover, for a lattice $L$ and  $n\in \Q$, $L(n)$ denotes 
the lattice whose underlying abelian group is the same as $L$ and the bilinear form is 
multiplied by $n$.  Basic facts concerning integral symmetric bilinear forms can be found in \cite{nikulin-sym}.

Since the lift $\tilde{f}$ commutes with the involution $\varepsilon$, the map it induces on the cohomology lattice 
 preserves sublattices $M$ and $N$. 
We put 
$$
\tdfm :=  \tilde f^* \rvert_M \mbox{ and } \tdfn := \tilde f^* \rvert_N  \, .
$$
Obviously, $\tdfn$ induces an isometry of the quadratic space $N \otimes \R$ of signature $(2,10)$ that 
preserves the original lattice $N =: N_{\Z} \subset N \otimes \R$
 and the Hodge structure of $N$. Since the latter is exactly given by an oriented positive 2-plane in $N \otimes \R$, we have
\[\tdfn \in O(N_{\Z})\cap (O(2)\times O(10)).\]
Since the right-hand side is a discrete subgroup in a compact group, we obtain 
the  following well-known fact. 
\begin{lemma}\label{Nfin}
The map $\tdfn$ is of finite order.
\end{lemma}

By Lemma \ref{Nfin}, the characteristic polynomial of $\tdfn$ is a product
\begin{equation} \label{eq-charpolfac}
\hhh_N(x)=\det (xI-\tdfn) = \prod_{i=1}^k \Phi_{n_i}(x) 
\end{equation} 
  of cyclotomic polynomials $\Phi_{n_i}(x)$ for a collection of 
positive integers $\{n_i :i=1, \ldots, k\}$ with $\sum_{i=1}^k \pphi(n_i) = 12$, where $\pphi(\cdot)$ stands for the Euler totient function. 
Obviously, the order of $\tdfn$ is just the least common multiple
\begin{equation} \label{eq-ord-tdfn}
\ord(\tdfn) = \lcm\{n_i, i = 1, \ldots,k\}.
\end{equation}

The integers $m$ for which $\pphi(m)\leq 12$ are as follows.
\begin{center}\begin{tabular}{c|l} \label{table-pphi}
$\pphi(m)$ & $m$ \\
\hline
$12$ & $13,21,26,28,36,42$ \\
$10$ & $11,22$             \\
$ 8$ & $15,16,20,24,30$    \\
$ 6$ & $7,9,14,18$         \\
$ 4$ & $5,8,10,12$         \\
$ 2$ & $3,4,6$             \\
$ 1$ & $1,2$               \\
\end{tabular}\end{center}

The proof of Thm~\ref{th1} will be based on the following lemmas. 
\begin{lemma} \label{lem1-new}
Let $\hhh_N(x)$ be the characteristic polynomial of the map $\tdfn$.

\noindent
{\rm (a)} The reduction $(\hhh_N(x) \bmod 2)$ is divisible by $(x^2+1)=(x+1)^2$.

\noindent
{\rm (b)} $\hhh_N(1) \hhh_N(-1)$ is either zero or a square in $\mathbb{Q}^{*}$.
\end{lemma}
\begin{proof} (a) Let us consider the action of $\tdfn$
on the reduction $N\otimes \mathbb{F}_2 \cong (1/2)N/N$.
Obviously, the reduction contains the 
10-dimensional $\tdfn$-invariant subspace 
\begin{equation} \label{eq-10dim}
N^*/N \subset(1/2)N/N, 
\end{equation} 
where $N^*/N$ is the discriminant group of $N$.
Thus the reduction $(\hhh_N(x) \bmod 2)$ is divisible by a  degree two polynomial.

In fact, there exists an $11$-dimensional canonical subspace $N^{*,+}$ of the reduction $(1/2)N/N$
containing $N^*/N$. We discuss as follows. 
The residue group $(1/2)N/N^*$ consists of four residue classes modulo $N^*$ and $N^*$ 
has the property that for all $y\in N^*$, $(y,y)\in \Z$ (namely $\delta (N)=0$ in Nikulin's notation.) Hence, the induced quadratic form $(1/2)N/N^*\rightarrow \Q/\Z$ is well-defined.
Among the four residue classes, there exists a unique nonzero element whose form value is $1/2$,
which corresponds to $N^{*,+}$. (The idea of this proof parallels \cite{allcock}). Thus we obtain
(a). 

(b) Claim follows from \cite[Proposition~5.1]{BF}. 
\end{proof} 

\begin{lemma} \label{lem-yuya}
The order $\mathrm{ord}(f_N)$ cannot be one of the integers $11, 22, 35, 70$. 
\end{lemma}
\begin{proof} We put $p_M(x) = \sum_{i=0}^{10} a_i x^i$ (resp. $p_{N^*/N}(x)$) to denote the characteristic polynomial of $\tdfm$ on $M$ (resp. of the map induced by $\tdfn$ on the discriminant group $N^*/N$).
Since $M \oplus N$ is a finite index sublattice of the unimodular lattice $H^2(\tilde{S}, \Z)$
we have a canonical isomorphism 
of the discriminant groups
$$M^*/M \cong N^*/N \, , $$
which is $(\tdfm,\tdfn)$-equivariant. 
Moreover, from $M^*/M = \frac{1}{2}M/M$ and the inclusion \eqref{eq-10dim}  we infer
\begin{equation} \label{eq-pnstarn}
p_{N^*/N}(x) = (p_M(x)  \bmod 2) \quad \mbox{and} \quad p_{N^*/N}(x) \divides    (p_N(x) \bmod 2) .   
\end{equation}

Assume $\mathrm{ord}(f_N) \in \{35, 70\}$. Then, by \eqref{eq-ord-tdfn} and the table  on p.~\pageref{table-pphi}  we have 
\begin{equation} \label{eq-pnx}
p_N(x) = \Phi_{5 l_1}(x) \Phi_{7 l_2}(x) \Phi_{l_3}(x) \Phi_{l_4}(x) \quad \mbox{ with }  l_1, l_2, l_3, l_4  \in \{1,2\} \, .
\end{equation}
Thus \eqref{eq-pnstarn} implies that 
$$
(p_M(x)  \bmod 2) = p_{N^*/N}(x) = F_5(x) F_7(x) = x^{10} + x^8 + x^6 + x^5 + x^4 + x^2 + 1 , 
$$
so the coefficient $a_5$ and the sums $a_0 + a_2 + a_4$, $a_6 + a_8 + a_{10}$ are  odd integers.
In particular, 
the   polynomial $p_M(x)$ is self-reciprocal (see e.g. \cite{cantat}), so the product  $p_M(1)p_M(-1)$
can be expressed as
\begin{eqnarray*}
 \bigl(\sum_{i:even} a_i\bigr)^2 - \bigl(\sum_{i:odd}a_i\bigr)^2 &=& (2(a_0 + a_2 + a_4))^2 - (2(a_1 + a_3) + a_5)^2   \\ 
                 &=&   (2 \bmod 4)^2 - (1 \bmod 2)^2  
\end{eqnarray*}
Thus  $p_M(1)p_M(-1)  \equiv 3 \bmod 8$, which contradicts \cite[Proposition~5.1]{BF}. 

If  $\mathrm{ord}(f_N) \in \{11, 22\}$, then the factorization of $p_N(x)$ is as follows
\begin{equation} \label{eq-pnx11}
p_N(x) = \Phi_{11 l_1}(x) \Phi_{l_2}(x) \Phi_{l_3}(x) \quad \mbox{ with }  l_1, l_2, l_3  \in \{1,2\} \, .
\end{equation}
Thus we have 
$(p_M(x)  \bmod 2) = F_{11}(x) $ and, as in the previous case we obtain  $p_M(1)p_M(-1)  \equiv 3 \bmod 8$.
\end{proof}

\begin{remark}
The proof of Lemma~\ref{lem-yuya} shows that the characteristic polynomial  $p_M(x) = \sum_{i=0}^{10} a_i x^i$ of $\tdfm$ cannot be a polynomial such that 
the coefficient $a_5$ and the sums $a_0 + a_2 + a_4 = a_6 + a_8 + a_{10}$ are  odd integers, i.e. the modulo 2 reduction $(p_M(x)  \bmod 2)$ cannot be one of the polynomials
$$F_5(x) F_7(x), F_{11}(x), 
F_3(x)^5, F_3(x)^2 F_9(x), F_3(x) F_5(x)^2, F_3(x) F_{15}(x) \, .$$
\end{remark}

The next lemma rules out the first two lines of the table on p.~\pageref{table-pphi}. 
It is also of use in the next section.
\begin{lemma} \label{lem12}
If 
$\Phi_{m}(x)$ comes up in the factorization \eqref{eq-charpolfac}, then
$$
\pphi(m) < 10 \, .
$$
\end{lemma}
\begin{proof}
If $\pphi(m)=12=\rank N$, then
the characteristic polynomial  $\hhh_N(x)$ equals the cyclotomic polynomial $\Phi_m(x)$.
 The decomposition of $\Phi_m(x) \bmod 2$ into irreducible factors is given in the table below:
\begin{center}\begin{tabular}{l|l}
$m$ & irreducible decomposition of $\Phi_m \bmod 2$ \\
\hline
42,21 & $(x^6+x^4+x^2+x+1)(x^6+x^5+x^4+x^2+1)$ \\
36 & $(x^6 + x^3 + 1)^2$ \\
28 & $(x^3 + x + 1)^2 (x^3 + x^2 + 1)^2$ \\
26,13 & $x^{12} + x^{11} + \cdots + x + 1$ \\
\end{tabular}\end{center}
Thus  $\pphi(m) < 12$ by Lemma~\ref{lem1-new}.a. 

Suppose that $\pphi(m)=10$, namely $m=11,22$. Then $\mathrm{ord}(f_N)$ is a multiple 
of 11, and taking some power we get a contradiction to Lemma \ref{lem-yuya}.
\end{proof}

After these preparations we can give the proof of Thm~\ref{th1}.
\begin{proof}[Proof of Theorem~\ref{th1}] 
Let $p(x) = \prod_{i=1}^k \Phi_{n_i}(x)$ be a product of cyclotomic polynomials for some $n_1$, $\ldots$ $n_k \in \mathbb{N}$ such that   $\sum_{i=1}^k \pphi(n_i) = 12$
and $\pphi(n_i) < 10$ for $i=1, \ldots, k$. If $p(x)$ is the characteristic polynomial of the  map  $\tdfn$ induced by an automorphism of an Enriques surface, then it satisfies
the conditions (a),(b) of Lemma~\ref{lem1-new} and the order of   $\tdfn$  is given by \eqref{eq-ord-tdfn}. 
An enumeration of all cases (by hand or by a help of a computer) and  Lemma~\ref{lem-yuya} show that 
$n =  \lcm\{n_i, i = 1, \ldots,k\}$ is one of the integers that 
appear in  Theorem~\ref{th1}.
\end{proof}

\begin{example} \label{rem-finite-ohashi}
To illustrate Theorem \ref{th1}, we shall give a classification of the case when the order of $f\in \mathrm{Aut}(S)$
is finite. Such automorphisms were classified in \cite{MO1, ohashi-birep}.
The following table gives the complete classification of the pair $(\mathrm{ord} (f),\mathrm{ord} (f_N))$.
\begin{center}\begin{tabular}{c|c|c|c|c|c|c|c}
$\mathrm{ord} (f)$ & 1&2&3&4&5&6&8\\ \hline
$\mathrm{ord} (f_N)$ & 1,2&1,2&3,6&1,2,4&5,10&3,6&4,8\\
\end{tabular}\end{center}
We list here nontrivial examples 
exhibiting the pair $(\mathrm{ord} (f),\mathrm{ord} (f_N))$ and leave the proofs to the reader.
The pair $(2,1)$ is supplied by
No.~18 in \cite{IO}. When $\mathrm{ord} (f)=3$, $5$ or $6$, $f$ is semi-symplectic by 
\cite[Proposition 4.5]{MO1} (i.e. $f$ acts trivially on the space $H^0(S, \mathcal{O}(2K_S))$)
and we can use the symplectic lift to compute eigenvalues. 
Examples 1.1 and 1.2 of \cite{ohashi-birep} give the pairs $(4,4)$ and $(8,8)$.
Finally, Example 1.3 of \cite{ohashi-birep} provides the remaining 
possibilities for $\mathrm{ord} (f)=4$ or $8$. 

As Example~\ref{example-existence-3factors} shows, the order $\mathrm{ord} (f_N)$ is no longer bounded by $10$ when the order of $f \in \mathrm{Aut}(S)$ is infinite
 (see \eqref{eq-order15}). However, we have very few examples:
the question of determining the exact list of possible $\ord(\tdfn)$ remains open.
\end{example}

\section{Dynamical degrees} \label{sect-dynamical-degrees}

We maintain the notation of the previous section. 
For the convenience of the reader we recall the definition of the dynamical degree.  

Let $f \in \mathrm{Aut}(S)$.  The {\bf dynamical degree   $\lambda(f)$ of} $f$ is defined as 
the spectral radius of the map $f^*:\mathrm{Num}(S) \rightarrow  \mathrm{Num}(S)$.
One can show that the map $f^*$ has either none or exactly  two 
eigenvalues away from the unit circle in $\C$. If such two eigenvalues 
come up, they are real and reciprocal. Thus either $\lambda(f) = 1$
or it is the largest real eigenvalue of the map $f^*$ (for a precise discussion of the above notion and its properties see  \cite[$\S$.2.2.2]{cantat}, \cite{mcmullen11}, \cite{mcmullen16} and references therein).

After these preparations we are in position to  give the proof of Thm~\ref{mod2}
(c.f. \cite[proof of Lemma~4.1]{third}).  

\begin{proof}[Proof of Theorem~\ref{mod2}]

We put $p_M$ (resp. $p_N$, resp. $p_f$) to denote the characteristic polynomial of $\tdfm$ on $M$ (resp. $\tdfn$ on $N$, resp. 
 $f^*$ on $\mathrm{Num}(S)$). 
We assume that $\lambda(f) \neq 1$ and denote the minimal polynomial  of 
$\lambda(f)$ by $s_\lambda$.

By the first isomorphism in \eqref{eq-enriques-isomorphisms}, the action of  $f^*$ on the discriminant group 
of the lattice $\mathrm{Num}(S)(2)$ coincides with the action of $\tdfm$ on  the discriminant group   $M^*/M$.
From $M^*/M = \frac{1}{2}M/M$,   we obtain 
the equality of modulo 2 reductions: 
$$p_M \equiv p_f \bmod 2 .$$  
Moreover \eqref{eq-pnstarn}  yields:
$$(p_M  \bmod 2) \divides    (p_N \bmod 2) .$$   

Let $h$ be an irreducible factor of the reduction  $(s_\lambda \bmod 2)$. We have just shown $h$ appears also in the  factorization of $(p_N \bmod 2)$.
Then, from \eqref{eq-charpolfac} and Lemma~\ref{lem12},
$h$ divides $(\Phi_m \bmod 2)$ for certain $m$ such that $\pphi(m) \leq 8$.
Since $(\Phi_{2^e m} \bmod 2)$ is a power of the reduction $(\Phi_m \bmod 2)$, we may assume $m$ to be odd.
Hence, by  the table on p.~\pageref{table-pphi}, we have  $m = 1,3,5,7,9,15$.
Thus either $h = F_m$, with $m=1,3,5,9$ or $h$ appears in the factorization of  $F_m$, where $m=7,15$.
But, for $m=7,15$ we have $F_m = F_{m,1} \cdot F_{m,2}$ where
\begin{align*}
F_{7,1} :=  (x^3 + x + 1),                           
& \quad  F_{7,2}  := (x^3 + x^2 + 1),   \\
F_{15,1} :=   (x^4 + x + 1),            
& \quad F_{15,2}  := (x^4 + x^3 + 1).
\end{align*}
Being a Salem polynomial, $s_\lambda$ is self-reciprocal, and so is its modulo $2$ reduction.
Since the polynomials $F_{7,1}$ and $F_{7,2}$ are not self-reciprocal,
their multiplicities in $(s_\lambda \bmod 2)$ should coincide with those of their reciprocal counterparts. 
The same holds for $F_{15,1}$ and $F_{15,2}$.
This completes the proof.
\end{proof}

It is natural to ask which of the 
six factors given in Thm~\ref{mod2}
do appear in modulo 2 reductions of minimal polynomials of dynamical degrees
of automorphisms of Enriques surfaces. We do not know whether the polynomials $F_{7}(x)$, $F_{9}(x)$, 
$F_{15}(x)$ 
are realized by automorphisms of Enriques surfaces. To answer the question whether $F_m$ where $m=1,3,5$ come up in $(s_{\lambda} \bmod 2)$,
we analyze some 
automorphisms constructed in \cite{dolgachev16}.

\begin{example}  \label{example-existence-3factors}
(a) By \cite[Sect.~4.5, Table~2]{dolgachev16} there exists an Enriques surface $S$ of Hesse type
and $f \in \mathrm{Aut}(S)$ such that the characteristic polynomial $p_{f^*}(x)$ of $f^* \in \mathrm{Aut}(\mathrm{Num}(S))$ 
equals\footnote{There is a misprint in \cite[Sect.~4.5, Table~2]{dolgachev16}: the terms $x^4$, $x^6$ appear with coefficient $6$ in  $p_{f^*}(x)$. 
In \eqref{eq-f5f3} we give the correct formula, but the misprint is irrelevant for us because we consider modulo 2 reduction.}
\begin{equation} \label{eq-f5f3}
p_{f^*}(x) = x^{10} - 6x^9 - 7x^8 - 9x^7 - 6x^6 - 10x^5 - 6x^4 - 9x^3 - 7x^2 - 6x +1 .
\end{equation}
One can  check  $p_{f^*}(x) \in \Z[x]$ is irreducible and  we have
$$
(p_{f^*}(x)  \bmod 2) = F_5(x) \cdot F_3(x) \cdot F_1^4(x).
$$ 
In particular, \eqref{eq-ord-tdfn} combined with the proof of  Thm~\ref{mod2} yields that 
\begin{equation} \label{eq-order15}
15 \divides \ord(\tdfn).
\end{equation}

\noindent
(b)  According to \cite[Sect.~3.2, Case $m = 4$]{dolgachev16} there exists an Enriques surface $S$ and $f \in \mathrm{Aut}(S)$ 
such that $p_{f^*}(x)$ is divisible by the following 
 polynomial
\begin{equation} \label{eq-dolgtocheck}
x^8 - 165x^7 + 223x^6 - 59x^5 - 133x^4 - 59x^3 + 223x^2 - 165x + 1 \, .
\end{equation}
Since the modulo 2 reduction of the above polynomial factors as the product $(F_3(x) \cdot F_9(x))$, this would imply that $F_9(x)$
can also appear in the reduction of the minimal polynomial $s_{\lambda}$. 
Unfortunately, there is a misprint in \cite[Sect.~3.2, Case $m = 4$]{dolgachev16}: the term $x^4$ comes with the coefficient $(-144)$
(instead of  $(-133)$ as in \eqref{eq-dolgtocheck}). Thus the question whether the factor $F_9(x)$ is possible or not 
remains open.
\end{example}

The example below shows that even the direct potential refinement of Thm~\ref{mod2} (i.e. ruling out all/some of the factors  $F_7$, $F_9$, 
$F_{15}$ in the modulo 2 reduction of the characteristic polynomial)  
cannot lead to a necessary and sufficient condition 
for a Salem number to be the  dynamical degree of an automorphism  of an Enriques surface.

\begin{example}  \label{example-strange} Consider  the  Salem number $\lambda = 1.64558...$  given by the polynomial 
\begin{equation} \label{eq-mini-salem}
s_{\lambda} := x^{10}-x^9-x^8-x^2-x+1 \, .
\end{equation}
As one can easily check, we have
$$
(s_{\lambda}  \bmod 2) =  F_3(x) \cdot F_1^8(x).   
$$
Thus  Thm~\ref{mod2}  does not rule out the above number as  the  dynamical degree of an automorphism  of an Enriques surface.  But we can apply \cite[Theorem 6.1]{Gross--McMullen}. 
Suppose that the Salem number given by \eqref{eq-mini-salem} is the dynamical degree of an automorphism of an Enriques surface $S$. 
The lattice $\mbox{Num}(S)$ is of rank 10, so $s_{\lambda}$ is the full characteristic polynomial of the induced  automorphism on $\mbox{Num}(S)$. Also $\mathrm{Num}(S)$ is even unimodular, so by \cite[Theorem 6.1]{Gross--McMullen}  both $|s_{\lambda}(\pm 1)|$ must be squares, which is not the case. 
This contradiction shows that  \eqref{eq-mini-salem} cannot be 
the minimal polynomial of  the dynamical degree of an automorphism of an Enriques surface. 
\end{example}

Presently, the smallest known non-trivial dynamical degree is the one constructed by Dolgachev:
\begin{example} \label{Dolgachevexample}
By  \cite[Sect.~4.5, Table~2]{dolgachev16} there exists an Enriques surface $S$ of Hesse type
and $f \in \mathrm{Aut}(S)$ such that the characteristic polynomial $p_{f^*}(x)$ of $f^* \in \mathrm{Aut}(\mathrm{Num}(S))$
has the Salem polynomial
\begin{equation} \label{eq-mp-lambdaD}
x^4-x^3-2x^2-x+1 
\end{equation}
as a factor. The largest real root of the above polynomial is $$\lambda_D := 2.08101... \, .$$ This
gives the smallest known value of non-trivial  dynamical degree of an automorphism of an Enriques surface.
\end{example}

\begin{example}\label{MO15family}
As an another example, let us consider the Enriques surface $S$ in the family of \cite{MO15}.
In that paper, it is proved that $\mathrm{Aut}(S)\simeq C_2^{*4}\rtimes \mathfrak{S}_4$.
We here prove that every element in this group has dynamical degree at least $\lambda_D$.

In fact, \cite{MO15} exhibits 10 smooth rational curves $E_i,\ E_{ij}\ (i,j=1,\dots,4,\ i\neq j)$
on $S$
which form a rational basis of $\mathrm{Num}(S)$. We denote by $L$ the sublattice generated by them. The symmetric group $\mathfrak{S}_4$ 
acts on the indices of generators, while the generators $s_i$ of the four cyclic groups $C_2$ 
act by reflections in some divisors $G_i\in L$ of self-intersection $-2$. In particular, 
$\mathrm{Aut}(S)$ preserves $L$. We use the relations $(G_i, E_{kl})=0$ (for all $i,k,l$)
and $(G_i, l)\in 2\mathbb{Z}$ for all $l\in L$.

Now from $(G_i, E_{kl})=0$, the six-dimensional subspace $V$ generated by $E_{ij}$ 
is stable under $\mathrm{Aut}(S)$. Since it is negative-definite, we see that any Salem 
number on $S$ has degree at most four, the dimension of the complement to $V$. Thus
for any element in $\mathrm{Aut}(S)$, the characteristic polynomial $F$ decomposes into 
degree four (on $L_{\mathbb{C}}/V_{\mathbb{C}}$) and degree six (on $V_{\mathbb{C}}$).
Moreover, since $G_i$ intersects evenly with all $l\in L$, it acts trivially on $L/2L$. 
Hence the degree four part of $(F \bmod 2)$ decomposes either into linear factors
or has only one non-linear factor $x^2+x+1$, which arises when the residue class has 
order 3 in $\mathfrak{S}_4$. Looking through the list in the Appendix, we get the assertion.
(This is something unfortunate, although.)
\end{example}

We use Theorem~\ref{mod2} on low-degree Salem numbers that do not exceed $\lambda_D$.
\begin{remark} \label{rem-possibilities}
One can check that there are exactly 133 Salem numbers of degree $\leq 10$ up to Dolgachev's record $\lambda_D$.
(The list of small Salem numbers can be found in \cite{mossinghoff}.)
Oguiso's criterion \cite[Lemma~4.1]{third} shows that 28 among them cannot be dynamical degrees of automorphisms 
of Enriques surfaces. To go further, combining Theorem~\ref{mod2} and \cite{Gross--McMullen} we can prove that 65 more cannot be realized by automorphisms
of Enriques surfaces.
For the convenience of readers we  list the 133 Salem numbers in question, their minimal polynomials and their modulo 2 reductions in $\S.$\ref{sect-appendix}~Appendix. 
Those numbers excluded by modulo 2 reductions are marked ``impossible''. 
Non-existence 
that results from  \cite[Theorem 6.1]{Gross--McMullen} (see Example~\ref{example-strange}) is marked ``impossible (*)''.
(In some cases both apply.)

In conclusion, {\sl there remain $39$ Salem numbers as candidates for the minimal dynamical degree of automorphisms 
of Enriques surfaces.} 

On the arXiv\footnote{\texttt{https://arxiv.org/abs/1707.02563}}
we attach two lists, in plain text format, of the coefficients of the minimal polynomials of  
(1) these $39$ candidates, and 
(2) all $133$ Salem numbers of degree $\leq 10$ up to Dolgachev's record $\lambda_D$.
\end{remark}

Finally we can give the proof of Corollary~\ref{th2}.
\begin{proof}[Proof of Corollary~\ref{th2}]
There are only finitely many Salem numbers less than $\lambda = 1.35098...$ and of degree $\leq 10$
(there are none of degree $\leq 6$, one of degree $8$, and six of degree $10$).
By a straightforward calculation 
each of those numbers 
violates the condition of Thm~\ref{mod2}.

For explicit factorization see the first seven entries of the table in $\S.$\ref{sect-appendix}~Appendix.
\end{proof}

\noindent
{\bf Acknowledgement:} This project was completed during the workshop "New Trends in Arithmetic and Geometry of Algebraic Surfaces"
in BIRS (Banff, Canada). The authors would like to thank the organizers and the staff of BIRS
for very good working conditions. We thank the referee for inspiring comments and calling our attention to the paper \cite{BF}.

\vspace{1cm}

\section{Appendix: List of low-degree Salem numbers up to $\lambda_D$} \label{sect-appendix}

Below we list all Salem numbers of degree $\leq 10$ up to the dynamical degree $\lambda_D$ of Dolgachev's example
and show the minimal polynomial $s_\lambda$ and its mod 2 factorization.
Some of them are impossible by Thm~\ref{mod2} and marked ``impossible''.
Others marked ``impossible (*)'' are those excluded by \cite[Theorem 6.1]{Gross--McMullen}.

As noted in Remark \ref{rem-possibilities}, the list of the coefficients of the minimal polynomials are available in plain text format on the arXiv. 
{\tiny 
\begin{longtable}{rrlll}
\# & deg & value & minimal polynomial $s_\lambda$ \\ &&& factorization of $s_\lambda \bmod 2$ & conclusion \\ \hline
\endhead
 & 10 & 1.17628... & $x^{10}+x^9-x^7-x^6-x^5-x^4-x^3+x+1$ & \\* &&& $(x^5 + x^3 + x^2 + x + 1)  (x^5 + x^4 + x^3 + x^2 + 1)$ & impossible \\ \hline
 & 10 & 1.21639... & $x^{10}-x^6-x^5-x^4+1$ & \\* &&& $(x^5 + x^4 + x^2 + x + 1)  (x^5 + x^4 + x^3 + x + 1)$ & impossible \\ \hline
 & 10 & 1.23039... & $x^{10}-x^7-x^5-x^3+1$ & \\* &&& $x^{10} + x^7 + x^5 + x^3 + 1$ & impossible \\ \hline
 & 10 & 1.26123... & $x^{10}-x^8-x^5-x^2+1$ & \\* &&& $(x^2 + x + 1)  (x^8 + x^7 + x^6 + x^4 + x^2 + x + 1)$ & impossible \\ \hline
 & 8 & 1.28063... & $x^8-x^5-x^4-x^3+1$ & \\* &&& $x^8 + x^5 + x^4 + x^3 + 1$ & impossible \\ \hline
 & 10 & 1.29348... & $x^{10}-x^8-x^7+x^5-x^3-x^2+1$ & \\* &&& $(x^5 + x^2 + 1)  (x^5 + x^3 + 1)$ & impossible \\ \hline
 & 10 & 1.33731... & $x^{10}-x^9-x^5-x+1$ & \\* &&& $x^{10} + x^9 + x^5 + x + 1$ & impossible \\ \hline
1 & 10 & 1.35098... & $x^{10}-x^9-x^6+x^5-x^4-x+1$ & \\* &&& $(x^2 + x + 1)^2  (x^3 + x + 1)  (x^3 + x^2 + 1)$ &  \\ \hline
 & 8 & 1.35999... & $x^8-x^7+x^6-2x^5+x^4-2x^3+x^2-x+1$ & \\* &&& $x^8 + x^7 + x^6 + x^4 + x^2 + x + 1$ & impossible \\ \hline
 & 10 & 1.38363... & $x^{10}-x^9-x^7+x^6-x^5+x^4-x^3-x+1$ & \\* &&& $(x^5 + x^3 + x^2 + x + 1)  (x^5 + x^4 + x^3 + x^2 + 1)$ & impossible \\ \hline
2 & 6 & 1.40126... & $x^6-x^4-x^3-x^2+1$ & \\* &&& $(x^2 + x + 1)  (x^4 + x^3 + x^2 + x + 1)$ &  \\ \hline
3 & 8 & 1.42500... & $x^8-x^7-x^5+x^4-x^3-x+1$ & \\* &&& $(x^4 + x + 1)  (x^4 + x^3 + 1)$ &  \\ \hline
 & 10 & 1.43100... & $x^{10}-x^9-x^8+x^7-x^5+x^3-x^2-x+1$ & \\* &&& $(x^2 + x + 1)  (x^8 + x^5 + x^4 + x^3 + 1)$ & impossible \\ \hline
 & 10 & 1.44842... & $x^{10}-2x^9+2x^8-2x^7+x^6-x^5+x^4-2x^3+2x^2-2x+1$ & \\* &&& $(x^5 + x^4 + x^2 + x + 1)  (x^5 + x^4 + x^3 + x + 1)$ & impossible \\ \hline
4 & 8 & 1.45798... & $x^8-x^6-x^5-x^3-x^2+1$ & \\* &&& $(x + 1)^2  (x^6 + x^3 + 1)$ &  \\ \hline
 & 10 & 1.47235... & $x^{10}-x^9-x^6-x^4-x+1$ & \\* &&& $(x + 1)^6  (x^4 + x^3 + x^2 + x + 1)$ & impossible (*) \\ \hline
 & 10 & 1.47960... & $x^{10}-2x^8-2x^7+x^6+3x^5+x^4-2x^3-2x^2+1$ & \\* &&& $(x^5 + x^4 + x^2 + x + 1)  (x^5 + x^4 + x^3 + x + 1)$ & impossible \\ \hline
5 & 6 & 1.50613... & $x^6-x^5-x^3-x+1$ & \\* &&& $(x^2 + x + 1)^3$ &  \\ \hline
 & 10 & 1.51386... & $x^{10}-x^7-2x^6-x^5-2x^4-x^3+1$ & \\* &&& $x^{10} + x^7 + x^5 + x^3 + 1$ & impossible \\ \hline
 & 8 & 1.52306... & $x^8-x^7-x^6+x^4-x^2-x+1$ & \\* &&& $x^8 + x^7 + x^6 + x^4 + x^2 + x + 1$ & impossible \\ \hline
6 & 10 & 1.53292... & $x^{10}-x^9-x^8+x^5-x^2-x+1$ & \\* &&& $(x^4 + x^3 + x^2 + x + 1)  (x^6 + x^3 + 1)$ &  \\ \hline
 & 8 & 1.54719... & $x^8-2x^7+2x^6-3x^5+3x^4-3x^3+2x^2-2x+1$ & \\* &&& $x^8 + x^5 + x^4 + x^3 + 1$ & impossible \\ \hline
7 & 6 & 1.55603... & $x^6-x^5-x^4+x^3-x^2-x+1$ & \\* &&& $(x^3 + x + 1)  (x^3 + x^2 + 1)$ &  \\ \hline
8 & 6 & 1.58234... & $x^6-x^4-2x^3-x^2+1$ & \\* &&& $(x + 1)^6$ &  \\ \hline
 & 10 & 1.59070... & $x^{10}-2x^9+x^8-2x^6+3x^5-2x^4+x^2-2x+1$ & \\* &&& $(x^2 + x + 1)  (x^8 + x^7 + x^6 + x^4 + x^2 + x + 1)$ & impossible \\ \hline
 & 10 & 1.59700... & $x^{10}-x^8-x^7-x^6-x^5-x^4-x^3-x^2+1$ & \\* &&& $(x^2 + x + 1)^2  (x^6 + x^3 + 1)$ &  impossible (*) \\ \hline
 & 10 & 1.59866... & $x^{10}-2x^9+x^8-x^7+2x^6-3x^5+2x^4-x^3+x^2-2x+1$ & \\* &&& $(x^5 + x^2 + 1)  (x^5 + x^3 + 1)$ & impossible \\ \hline
9 & 8 & 1.60544... & $x^8-2x^7+x^6-x^4+x^2-2x+1$ & \\* &&& $(x^4 + x^3 + x^2 + x + 1)^2$ &  \\ \hline
 & 10 & 1.62501... & $x^{10}-x^9-x^8-x^7+x^6+x^5+x^4-x^3-x^2-x+1$ & \\* &&& $x^{10} + x^9 + x^8 + x^7 + x^6 + x^5 + x^4 + x^3 + x^2 + x + 1$ &  impossible (*) \\ \hline
 & 10 & 1.62754... & $x^{10}-2x^9+2x^7-x^6-x^5-x^4+2x^3-2x+1$ & \\* &&& $(x^5 + x^4 + x^2 + x + 1)  (x^5 + x^4 + x^3 + x + 1)$ & impossible \\ \hline
10 & 6 & 1.63557... & $x^6-2x^5+2x^4-3x^3+2x^2-2x+1$ & \\* &&& $x^6 + x^3 + 1$ &  \\ \hline
 & 8 & 1.64003... & $x^8-2x^6-x^5+x^4-x^3-2x^2+1$ & \\* &&& $x^8 + x^5 + x^4 + x^3 + 1$ & impossible \\ \hline
 & 10 & 1.64558... & $x^{10}-x^9-x^8-x^2-x+1$ & \\* &&& $(x + 1)^8  (x^2 + x + 1)$ &  impossible (*) \\ \hline
 & 10 & 1.65740... & $x^{10}+x^9-2x^7-4x^6-5x^5-4x^4-2x^3+x+1$ & \\* &&& $x^{10} + x^9 + x^5 + x + 1$ & impossible \\ \hline
11 & 8 & 1.66104... & $x^8-2x^7+x^6-x^5+x^4-x^3+x^2-2x+1$ & \\* &&& $(x^2 + x + 1)  (x^3 + x + 1)  (x^3 + x^2 + 1)$ &  \\ \hline
 & 10 & 1.66929... & $x^{10}-x^9-x^7-2x^5-x^3-x+1$ & \\* &&& $(x + 1)^2  (x^8 + x^7 + x^6 + x^4 + x^2 + x + 1)$ & impossible \\ \hline
 & 10 & 1.67310... & $x^{10}-2x^9+x^8-x^7+x^5-x^3+x^2-2x+1$ & \\* &&& $(x^5 + x^2 + 1)  (x^5 + x^3 + 1)$ & impossible \\ \hline
12 & 8 & 1.68491... & $x^8-x^7-x^6-x^2-x+1$ & \\* &&& $(x + 1)^2  (x^2 + x + 1)^3$ &  \\ \hline
 & 10 & 1.69017... & $x^{10}-x^9-2x^8+x^7+x^6-x^5+x^4+x^3-2x^2-x+1$ & \\* &&& $(x^5 + x^3 + x^2 + x + 1)  (x^5 + x^4 + x^3 + x^2 + 1)$ & impossible \\ \hline
13 & 8 & 1.69350... & $x^8-x^7-x^5-x^4-x^3-x+1$ & \\* &&& $(x^4 + x + 1)  (x^4 + x^3 + 1)$ &  \\ \hline
 & 10 & 1.71336... & $x^{10}-2x^9+x^8-2x^6+2x^5-2x^4+x^2-2x+1$ & \\* &&& $(x + 1)^{10}$ &  impossible (*) \\ \hline
14 & 4 & 1.72208... & $x^4-x^3-x^2-x+1$ & \\* &&& $x^4 + x^3 + x^2 + x + 1$ &  \\ \hline
 & 10 & 1.73694... & $x^{10}-x^9-x^7-x^6-x^5-x^4-x^3-x+1$ & \\* &&& $(x^5 + x^3 + x^2 + x + 1)  (x^5 + x^4 + x^3 + x^2 + 1)$ & impossible \\ \hline
 & 10 & 1.74492... & $x^{10}-2x^9+2x^8-3x^7+2x^6-3x^5+2x^4-3x^3+2x^2-2x+1$ & \\* &&& $x^{10} + x^7 + x^5 + x^3 + 1$ & impossible \\ \hline
 & 10 & 1.74601... & $x^{10}-x^9-x^8-x^7+2x^5-x^3-x^2-x+1$ & \\* &&& $(x + 1)^4  (x^3 + x + 1)  (x^3 + x^2 + 1)$ &  impossible (*) \\ \hline
 & 10 & 1.75173... & $x^{10}-2x^9+x^8-x^7+x^6-2x^5+x^4-x^3+x^2-2x+1$ & \\* &&& $(x + 1)^2  (x^8 + x^5 + x^4 + x^3 + 1)$ & impossible \\ \hline
 & 10 & 1.75309... & $x^{10}-x^8-x^7-2x^6-3x^5-2x^4-x^3-x^2+1$ & \\* &&& $(x^5 + x^2 + 1)  (x^5 + x^3 + 1)$ & impossible \\ \hline
 & 10 & 1.76015... & $x^{10}-2x^9+x^8-2x^7+2x^6-x^5+2x^4-2x^3+x^2-2x+1$ & \\* &&& $(x^2 + x + 1)  (x^8 + x^7 + x^6 + x^4 + x^2 + x + 1)$ & impossible \\ \hline
 & 10 & 1.76400... & $x^{10}-2x^9+x^7-x^5+x^3-2x+1$ & \\* &&& $x^{10} + x^7 + x^5 + x^3 + 1$ & impossible \\ \hline
 & 10 & 1.76690... & $x^{10}-2x^8-2x^7+x^5-2x^3-2x^2+1$ & \\* &&& $(x^2 + x + 1)  (x^4 + x + 1)  (x^4 + x^3 + 1)$ &  impossible (*) \\ \hline
 & 10 & 1.77056... & $x^{10}-3x^9+4x^8-5x^7+5x^6-5x^5+5x^4-5x^3+4x^2-3x+1$ & \\* &&& $(x^5 + x^3 + x^2 + x + 1)  (x^5 + x^4 + x^3 + x^2 + 1)$ & impossible \\ \hline
15 & 6 & 1.78164... & $x^6-x^5-x^4-x^2-x+1$ & \\* &&& $(x + 1)^4  (x^2 + x + 1)$ &  \\ \hline
 & 10 & 1.78840... & $x^{10}-x^9-2x^7-x^5-2x^3-x+1$ & \\* &&& $x^{10} + x^9 + x^5 + x + 1$ & impossible \\ \hline
 & 8 & 1.79607... & $x^8-x^7-x^6-x^4-x^2-x+1$ & \\* &&& $x^8 + x^7 + x^6 + x^4 + x^2 + x + 1$ & impossible \\ \hline
 & 10 & 1.79978... & $x^{10}-3x^8-3x^7+2x^6+5x^5+2x^4-3x^3-3x^2+1$ & \\* &&& $(x^5 + x^2 + 1)  (x^5 + x^3 + 1)$ & impossible \\ \hline
16 & 8 & 1.80017... & $x^8-3x^7+4x^6-5x^5+5x^4-5x^3+4x^2-3x+1$ & \\* &&& $(x^4 + x + 1)  (x^4 + x^3 + 1)$ &  \\ \hline
17 & 10 & 1.80501... & $x^{10}-2x^9+x^7-x^6+x^5-x^4+x^3-2x+1$ & \\* &&& $(x^2 + x + 1)^3  (x^4 + x^3 + x^2 + x + 1)$ &  \\ \hline
18 & 8 & 1.80978... & $x^8-x^7-2x^5-2x^3-x+1$ & \\* &&& $(x + 1)^2  (x^3 + x + 1)  (x^3 + x^2 + 1)$ &  \\ \hline
 & 8 & 1.81161... & $x^8-2x^7+x^5-x^4+x^3-2x+1$ & \\* &&& $x^8 + x^5 + x^4 + x^3 + 1$ & impossible \\ \hline
 & 10 & 1.82383... & $x^{10}-x^8-2x^7-2x^6-2x^5-2x^4-2x^3-x^2+1$ & \\* &&& $(x + 1)^{10}$ &  impossible (*) \\ \hline
19 & 10 & 1.82514... & $x^{10}-x^9-2x^8+x^6+x^5+x^4-2x^2-x+1$ & \\* &&& $(x^2 + x + 1)^2  (x^3 + x + 1)  (x^3 + x^2 + 1)$ &  \\ \hline
20 & 6 & 1.83107... & $x^6-2x^5+x^3-2x+1$ & \\* &&& $x^6 + x^3 + 1$ &  \\ \hline
21 & 8 & 1.83488... & $x^8-x^6-2x^5-3x^4-2x^3-x^2+1$ & \\* &&& $(x^4 + x^3 + x^2 + x + 1)^2$ &  \\ \hline
 & 10 & 1.84835... & $x^{10}-x^9-x^8-x^6-x^5-x^4-x^2-x+1$ & \\* &&& $(x^2 + x + 1)^5$ &  impossible (*) \\ \hline
 & 8 & 1.84959... & $x^8+x^7-x^6-4x^5-5x^4-4x^3-x^2+x+1$ & \\* &&& $x^8 + x^7 + x^6 + x^4 + x^2 + x + 1$ & impossible \\ \hline
 & 10 & 1.85312... & $x^{10}-2x^9+x^8-2x^7+2x^6-2x^5+2x^4-2x^3+x^2-2x+1$ & \\* &&& $(x + 1)^{10}$ &  impossible (*) \\ \hline
 & 10 & 1.85712... & $x^{10}-2x^9+x^8-x^7-x^5-x^3+x^2-2x+1$ & \\* &&& $(x^5 + x^2 + 1)  (x^5 + x^3 + 1)$ & impossible \\ \hline
 & 10 & 1.86264... & $x^{10}-3x^9+3x^8-x^7-3x^6+5x^5-3x^4-x^3+3x^2-3x+1$ & \\* &&& $x^{10} + x^9 + x^8 + x^7 + x^6 + x^5 + x^4 + x^3 + x^2 + x + 1$ &  impossible (*) \\ \hline
22 & 8 & 1.86406... & $x^8-x^7-2x^6+2x^4-2x^2-x+1$ & \\* &&& $(x + 1)^2  (x^3 + x + 1)  (x^3 + x^2 + 1)$ &  \\ \hline
 & 10 & 1.86876... & $x^{10}-x^9-2x^7-3x^5-2x^3-x+1$ & \\* &&& $x^{10} + x^9 + x^5 + x + 1$ & impossible \\ \hline
 & 10 & 1.87573... & $x^{10}-2x^9-x^8+3x^7-3x^5+3x^3-x^2-2x+1$ & \\* &&& $(x^5 + x^2 + 1)  (x^5 + x^3 + 1)$ & impossible \\ \hline
23 & 4 & 1.88320... & $x^4-2x^3+x^2-2x+1$ & \\* &&& $(x^2 + x + 1)^2$ &  \\ \hline
 & 10 & 1.88996... & $x^{10}-x^9-x^8-x^6-2x^5-x^4-x^2-x+1$ & \\* &&& $(x + 1)^2  (x^4 + x + 1)  (x^4 + x^3 + 1)$ &  impossible (*) \\ \hline
 & 10 & 1.89360... & $x^{10}-x^9-2x^8+x^6+x^4-2x^2-x+1$ & \\* &&& $(x + 1)^6  (x^4 + x^3 + x^2 + x + 1)$ &  impossible (*) \\ \hline
 & 10 & 1.89663... & $x^{10}-2x^9+x^7-x^6-x^4+x^3-2x+1$ & \\* &&& $(x + 1)^4  (x^6 + x^3 + 1)$ &  impossible (*) \\ \hline
 & 10 & 1.89910... & $x^{10}-2x^9+x^6-x^5+x^4-2x+1$ & \\* &&& $(x^5 + x^4 + x^2 + x + 1)  (x^5 + x^4 + x^3 + x + 1)$ & impossible \\ \hline
 & 10 & 1.90562... & $x^{10}-x^8-2x^7-3x^6-3x^5-3x^4-2x^3-x^2+1$ & \\* &&& $(x^3 + x + 1)  (x^3 + x^2 + 1)  (x^4 + x^3 + x^2 + x + 1)$ &  impossible (*) \\ \hline
 & 10 & 1.90830... & $x^{10}-x^9-x^8-x^7-x^5-x^3-x^2-x+1$ & \\* &&& $(x^2 + x + 1)  (x^8 + x^5 + x^4 + x^3 + 1)$ & impossible \\ \hline
 & 10 & 1.91112... & $x^{10}-2x^8-2x^7-x^6-x^5-x^4-2x^3-2x^2+1$ & \\* &&& $(x^5 + x^4 + x^2 + x + 1)  (x^5 + x^4 + x^3 + x + 1)$ & impossible \\ \hline
 & 10 & 1.91445... & $x^{10}-x^9-x^7-3x^6-x^5-3x^4-x^3-x+1$ & \\* &&& $(x^5 + x^3 + x^2 + x + 1)  (x^5 + x^4 + x^3 + x^2 + 1)$ & impossible \\ \hline
24 & 8 & 1.91649... & $x^8-x^7-x^6-x^5-x^3-x^2-x+1$ & \\* &&& $(x + 1)^4  (x^4 + x^3 + x^2 + x + 1)$ &  \\ \hline
 & 8 & 1.92062... & $x^8-3x^7+3x^6-2x^5+x^4-2x^3+3x^2-3x+1$ & \\* &&& $x^8 + x^7 + x^6 + x^4 + x^2 + x + 1$ & impossible \\ \hline
 & 10 & 1.92606... & $x^{10}-x^9-x^8-x^7-x^6+x^5-x^4-x^3-x^2-x+1$ & \\* &&& $x^{10} + x^9 + x^8 + x^7 + x^6 + x^5 + x^4 + x^3 + x^2 + x + 1$ &  impossible (*) \\ \hline
25 & 8 & 1.92678... & $x^8-2x^6-2x^5-x^4-2x^3-2x^2+1$ & \\* &&& $(x^2 + x + 1)^4$ &  \\ \hline
 & 10 & 1.92990... & $x^{10}-x^9-x^8-2x^7+x^6+x^4-2x^3-x^2-x+1$ & \\* &&& $(x + 1)^2  (x^4 + x + 1)  (x^4 + x^3 + 1)$ & impossible (*)  \\ \hline
 & 10 & 1.93231... & $x^{10}-2x^9+2x^8-4x^7+3x^6-5x^5+3x^4-4x^3+2x^2-2x+1$ & \\* &&& $(x^5 + x^4 + x^2 + x + 1)  (x^5 + x^4 + x^3 + x + 1)$ & impossible \\ \hline
 & 10 & 1.93295... & $x^{10}-3x^9+3x^8-2x^7+x^5-2x^3+3x^2-3x+1$ & \\* &&& $(x^4 + x^3 + x^2 + x + 1)  (x^6 + x^3 + 1)$ &  impossible (*) \\ \hline
 & 10 & 1.93637... & $x^{10}-2x^9+x^5-2x+1$ & \\* &&& $(x^2 + x + 1)  (x^4 + x + 1)  (x^4 + x^3 + 1)$ &  impossible (*) \\ \hline
 & 10 & 1.94005... & $x^{10}-2x^9+x^8-x^7-x^6-x^4-x^3+x^2-2x+1$ & \\* &&& $(x + 1)^2  (x^8 + x^5 + x^4 + x^3 + 1)$ & impossible \\ \hline
26 & 6 & 1.94685... & $x^6-x^5-x^4-x^3-x^2-x+1$ & \\* &&& $(x^3 + x + 1)  (x^3 + x^2 + 1)$ &  \\ \hline
 & 10 & 1.94998... & $x^{10}-x^9-2x^8-x^7+x^6+3x^5+x^4-x^3-2x^2-x+1$ & \\* &&& $(x^5 + x^3 + x^2 + x + 1)  (x^5 + x^4 + x^3 + x^2 + 1)$ & impossible \\ \hline
 & 8 & 1.95530... & $x^8-2x^7-x^5+3x^4-x^3-2x+1$ & \\* &&& $x^8 + x^5 + x^4 + x^3 + 1$ & impossible \\ \hline
27 & 6 & 1.96355... & $x^6-2x^5-x^4+3x^3-x^2-2x+1$ & \\* &&& $(x^2 + x + 1)  (x^4 + x^3 + x^2 + x + 1)$ &  \\ \hline
 & 10 & 1.97209... & $x^{10}-2x^9-x^6+3x^5-x^4-2x+1$ & \\* &&& $(x^5 + x^4 + x^2 + x + 1)  (x^5 + x^4 + x^3 + x + 1)$ & impossible \\ \hline
28 & 6 & 1.97481... & $x^6-2x^5+x^4-2x^3+x^2-2x+1$ & \\* &&& $(x + 1)^6$ &  \\ \hline
29 & 6 & 1.98779... & $x^6-2x^4-3x^3-2x^2+1$ & \\* &&& $x^6 + x^3 + 1$ &  \\ \hline
30 & 8 & 1.99400... & $x^8-2x^7+x^6-2x^5+x^4-2x^3+x^2-2x+1$ & \\* &&& $(x^4 + x^3 + x^2 + x + 1)^2$ &  \\ \hline
 & 10 & 1.99703... & $x^{10}-x^9-x^8-x^7-x^6-x^5-x^4-x^3-x^2-x+1$ & \\* &&& $x^{10} + x^9 + x^8 + x^7 + x^6 + x^5 + x^4 + x^3 + x^2 + x + 1$ &  impossible (*) \\ \hline
31 & 10 & 1.99852... & $x^{10}-2x^9+x^8-2x^7+x^6-2x^5+x^4-2x^3+x^2-2x+1$ & \\* &&& $(x + 1)^2  (x^2 + x + 1)^4$ &  \\ \hline
 & 10 & 2.00145... & $x^{10}-x^9-2x^8-x^7+x^6+2x^5+x^4-x^3-2x^2-x+1$ & \\* &&& $(x + 1)^4  (x^2 + x + 1)^3$ & impossible (*)  \\ \hline
 & 10 & 2.00289... & $x^{10}-3x^9+3x^8-3x^7+3x^6-3x^5+3x^4-3x^3+3x^2-3x+1$ & \\* &&& $x^{10} + x^9 + x^8 + x^7 + x^6 + x^5 + x^4 + x^3 + x^2 + x + 1$ &  impossible (*) \\ \hline
 & 10 & 2.00573... & $x^{10}-2x^9-2x+1$ & \\* &&& $(x + 1)^2  (x^4 + x^3 + x^2 + x + 1)^2$ & impossible (*)  \\ \hline
 & 10 & 2.00624... & $x^{10}-x^8-3x^7-3x^6-4x^5-3x^4-3x^3-x^2+1$ & \\* &&& $(x + 1)^2  (x^8 + x^5 + x^4 + x^3 + 1)$ & impossible \\ \hline
 & 10 & 2.00947... & $x^{10}-2x^9+x^8-x^7-2x^6+x^5-2x^4-x^3+x^2-2x+1$ & \\* &&& $(x^5 + x^2 + 1)  (x^5 + x^3 + 1)$ & impossible \\ \hline
32 & 8 & 2.01128... & $x^8-3x^7+3x^6-3x^5+3x^4-3x^3+3x^2-3x+1$ & \\* &&& $(x^2 + x + 1)  (x^6 + x^3 + 1)$ &  \\ \hline
 & 10 & 2.01335... & $x^{10}-x^9-2x^7-2x^6-3x^5-2x^4-2x^3-x+1$ & \\* &&& $x^{10} + x^9 + x^5 + x + 1$ & impossible \\ \hline
 & 10 & 2.01488... & $x^{10}-2x^8-2x^7-2x^6-3x^5-2x^4-2x^3-2x^2+1$ & \\* &&& $(x^2 + x + 1)  (x^4 + x + 1)  (x^4 + x^3 + 1)$ &  impossible (*) \\ \hline
 & 10 & 2.01600... & $x^{10}-3x^9+2x^8+x^7-3x^6+3x^5-3x^4+x^3+2x^2-3x+1$ & \\* &&& $(x^5 + x^3 + x^2 + x + 1)  (x^5 + x^4 + x^3 + x^2 + 1)$ & impossible \\ \hline
 & 10 & 2.01671... & $x^{10}-2x^9-x^8+2x^7+x^6-3x^5+x^4+2x^3-x^2-2x+1$ & \\* &&& $(x^3 + x + 1)  (x^3 + x^2 + 1)  (x^4 + x^3 + x^2 + x + 1)$ &  impossible (*) \\ \hline
33 & 8 & 2.02202... & $x^8-2x^7-2x+1$ & \\* &&& $(x + 1)^8$ &  \\ \hline
 & 10 & 2.02344... & $x^{10}-2x^9-x^7+2x^6-x^5+2x^4-x^3-2x+1$ & \\* &&& $x^{10} + x^7 + x^5 + x^3 + 1$ & impossible \\ \hline
 & 10 & 2.02739... & $x^{10}-x^9-x^8-x^7-x^6-2x^5-x^4-x^3-x^2-x+1$ & \\* &&& $(x + 1)^2  (x^2 + x + 1)^2  (x^4 + x^3 + x^2 + x + 1)$ &  impossible (*) \\ \hline
34 & 8 & 2.03064... & $x^8-x^7-3x^5-x^4-3x^3-x+1$ & \\* &&& $(x^4 + x + 1)  (x^4 + x^3 + 1)$ &  \\ \hline
 & 10 & 2.03298... & $x^{10}-2x^9+x^6-3x^5+x^4-2x+1$ & \\* &&& $(x^5 + x^4 + x^2 + x + 1)  (x^5 + x^4 + x^3 + x + 1)$ & impossible \\ \hline
 & 10 & 2.03579... & $x^{10}-2x^9-x^6+2x^5-x^4-2x+1$ & \\* &&& $(x + 1)^6  (x^2 + x + 1)^2$ &  impossible (*) \\ \hline
 & 10 & 2.03890... & $x^{10}-x^9-2x^8-2x^7+2x^6+3x^5+2x^4-2x^3-2x^2-x+1$ & \\* &&& $x^{10} + x^9 + x^5 + x + 1$ & impossible \\ \hline
35 & 6 & 2.04249... & $x^6-3x^5+3x^4-3x^3+3x^2-3x+1$ & \\* &&& $(x^3 + x + 1)  (x^3 + x^2 + 1)$ &  \\ \hline
 & 10 & 2.04414... & $x^{10}-x^9-2x^8-x^7+x^6+x^5+x^4-x^3-2x^2-x+1$ & \\* &&& $(x^5 + x^3 + x^2 + x + 1)  (x^5 + x^4 + x^3 + x^2 + 1)$ & impossible \\ \hline
 & 10 & 2.04776... & $x^{10}-3x^9+3x^8-3x^7+2x^6-x^5+2x^4-3x^3+3x^2-3x+1$ & \\* &&& $(x^2 + x + 1)  (x^8 + x^5 + x^4 + x^3 + 1)$ & impossible \\ \hline
 & 10 & 2.04817... & $x^{10}-x^9-2x^8-x^5-2x^2-x+1$ & \\* &&& $x^{10} + x^9 + x^5 + x + 1$ & impossible \\ \hline
36 & 8 & 2.04952... & $x^8-x^7-x^6-x^5-2x^4-x^3-x^2-x+1$ & \\* &&& $(x + 1)^4  (x^4 + x^3 + x^2 + x + 1)$ &  \\ \hline
 & 10 & 2.05286... & $x^{10}-x^9-x^8-2x^7-x^5-2x^3-x^2-x+1$ & \\* &&& $(x^4 + x^3 + x^2 + x + 1)  (x^6 + x^3 + 1)$ & impossible (*)  \\ \hline
 & 10 & 2.05353... & $x^{10}-2x^9-x^8+2x^7-x^5+2x^3-x^2-2x+1$ & \\* &&& $(x^2 + x + 1)  (x^8 + x^7 + x^6 + x^4 + x^2 + x + 1)$ & impossible \\ \hline
 & 10 & 2.05523... & $x^{10}-x^9-x^8-x^7-x^6-3x^5-x^4-x^3-x^2-x+1$ & \\* &&& $x^{10} + x^9 + x^8 + x^7 + x^6 + x^5 + x^4 + x^3 + x^2 + x + 1$ &   impossible (*) \\ \hline
37 & 10 & 2.05631... & $x^{10}-2x^9-x^7+x^6+x^5+x^4-x^3-2x+1$ & \\* &&& $(x^2 + x + 1)^3  (x^4 + x^3 + x^2 + x + 1)$ &  \\ \hline
 & 10 & 2.05819... & $x^{10}-x^9-3x^7-x^6-3x^5-x^4-3x^3-x+1$ & \\* &&& $(x^5 + x^3 + x^2 + x + 1)  (x^5 + x^4 + x^3 + x^2 + 1)$ & impossible \\ \hline
38 & 8 & 2.06017... & $x^8-3x^7+2x^6+x^5-3x^4+x^3+2x^2-3x+1$ & \\* &&& $(x^4 + x + 1)  (x^4 + x^3 + 1)$ &  \\ \hline
 & 10 & 2.06226... & $x^{10}-2x^8-3x^7-2x^6-x^5-2x^4-3x^3-2x^2+1$ & \\* &&& $x^{10} + x^7 + x^5 + x^3 + 1$ & impossible \\ \hline
 & 10 & 2.06420... & $x^{10}-2x^9+x^8-2x^7-x^5-2x^3+x^2-2x+1$ & \\* &&& $(x^2 + x + 1)  (x^8 + x^7 + x^6 + x^4 + x^2 + x + 1)$ & impossible \\ \hline
 & 10 & 2.06557... & $x^{10}-x^8-3x^7-4x^6-5x^5-4x^4-3x^3-x^2+1$ & \\* &&& $(x^5 + x^2 + 1)  (x^5 + x^3 + 1)$ & impossible \\ \hline
 & 8 & 2.06972... & $x^8-2x^7+x^5-3x^4+x^3-2x+1$ & \\* &&& $x^8 + x^5 + x^4 + x^3 + 1$ & impossible \\ \hline
 & 10 & 2.07416... & $x^{10}-3x^9+4x^8-6x^7+6x^6-7x^5+6x^4-6x^3+4x^2-3x+1$ & \\* &&& $x^{10} + x^9 + x^5 + x + 1$ & impossible \\ \hline
39 & 4 & 2.08101... & $x^4-x^3-2x^2-x+1$ & \\* &&& $(x + 1)^2  (x^2 + x + 1)$ & $\exists$ example \\ \hline
\end{longtable}
}

\end{document}